\documentclass[twoside]{amsart}

\usepackage{amsmath,amsfonts,amsthm,mathrsfs}
\usepackage{amssymb}

\usepackage[unicode,bookmarks,colorlinks]{hyperref}
\usepackage[usenames,dvipsnames]{xcolor}
\hypersetup{citecolor=NavyBlue,linkcolor=BrickRed,urlcolor=Orange}

\usepackage[alphabetic,initials]{amsrefs}

\usepackage{enumitem}

\newenvironment{aenumerate}{%
	\begin{enumerate}[label=(\alph{*}), ref=(\alph{*})]
}{%
	\end{enumerate}%
}

\usepackage{chngcntr}

\usepackage{tikz}
\usepackage{tikz-cd}
\tikzset{commutative diagrams/arrow style=math font}

\newcommand{\Dmod}{\mathscr{D}}
\newcommand{\Mmod}{\mathcal{M}}

\newcommand{\shT}{\mathscr{T}}

\newcommand{\derR}{\mathbf{R}}

\newcommand{\decal}[1]{\lbrack #1 \rbrack}

\newcommand{\shH}{\mathcal{H}}

\newcommand{\tensor}{\otimes}

\newcommand{\shHom}{\mathcal{H}\hspace{-1pt}\mathit{om}}

\newcommand{\ZZ}{\mathbb{Z}}

\newcommand{\CC}{\mathbb{C}}

\newcommand{\menge}[2]{\bigl\{ \thinspace #1 \thinspace\thinspace \big\vert%
\thinspace\thinspace #2 \thinspace \bigr\}}

\DeclareMathOperator{\im}{im}

\DeclareMathOperator{\id}{id}

\DeclareMathOperator{\Supp}{Supp}

\DeclareMathOperator{\Sym}{Sym}
\DeclareMathOperator{\gr}{gr}

\DeclareMathOperator{\Var}{Var}
\DeclareMathOperator{\Ch}{Ch}
\DeclareMathOperator{\Alb}{Alb}
\DeclareMathOperator{\Pic}{Pic}

\newcommand{\define}[1]{\emph{#1}}

\newcommand{\shf}[1]{\mathscr{#1}}
\newcommand{\OX}{\shf{O}_X}
\newcommand{\OmX}{\Omega_X}

\newcommand{\fu}{f^{\ast}}
\newcommand{\fl}{f_{\ast}}

\newcommand{\pu}{p^{\ast}}

\newcommand{\piu}{\pi^{\ast}}

\newcommand{\hu}{h^{\ast}}
\newcommand{\varphiu}{\varphi^{\ast}}
\newcommand{\varphil}{\varphi_{\ast}}
\newcommand{\hl}{h_{\ast}}
\newcommand{\fp}{f_{+}}
\newcommand{\hp}{h_{+}}

\newcommand{\shF}{\shf{F}}

\newcommand{\shE}{\shf{E}}

\newcommand{\shO}{\shf{O}}

\makeatletter
\let\@@seccntformat\@seccntformat
\renewcommand*{\@seccntformat}[1]{%
  \expandafter\ifx\csname @seccntformat@#1\endcsname\relax
    \expandafter\@@seccntformat
  \else
    \expandafter
      \csname @seccntformat@#1\expandafter\endcsname
  \fi
    {#1}%
}
\newcommand*{\@seccntformat@subsection}[1]{%
  \textbf{\csname the#1\endcsname.}
}
\makeatother

\makeatletter
\let\@paragraph\paragraph
\renewcommand*{\paragraph}[1]{%
	\vspace{0.3\baselineskip}%
	\@paragraph{\textit{#1}}%
}
\makeatother

\counterwithin{equation}{subsection}
\counterwithout{subsection}{section}
\counterwithin{figure}{subsection}

\newtheorem{theorem}[equation]{Theorem}
\newtheorem*{theorem*}{Theorem}
\newtheorem{lemma}[equation]{Lemma}
\newtheorem*{lemma*}{Lemma}
\newtheorem{corollary}[equation]{Corollary}
\newtheorem{proposition}[equation]{Proposition}
\newtheorem*{proposition*}{Proposition}
\newtheorem{conjecture}[equation]{Conjecture}

\theoremstyle{definition}

\newtheorem*{definition*}{Definition}
\theoremstyle{remark}

\newtheorem*{example*}{Example}
\newtheorem*{problem*}{Problem}

\theoremstyle{plain}

\newcommand{\theoremref}[1]{\hyperref[#1]{Theorem~\ref*{#1}}}
\newcommand{\lemmaref}[1]{\hyperref[#1]{Lemma~\ref*{#1}}}
\newcommand{\definitionref}[1]{\hyperref[#1]{Definition~\ref*{#1}}}
\newcommand{\propositionref}[1]{\hyperref[#1]{Proposition~\ref*{#1}}}
\newcommand{\conjectureref}[1]{\hyperref[#1]{Conjecture~\ref*{#1}}}
\newcommand{\corollaryref}[1]{\hyperref[#1]{Corollary~\ref*{#1}}}
\newcommand{\exampleref}[1]{\hyperref[#1]{Example~\ref*{#1}}}

\newcounter{intro}

\newtheorem{intro-conjecture}[intro]{Conjecture}
\newtheorem{intro-corollary}[intro]{Corollary}
\newtheorem{intro-theorem}[intro]{Theorem}

\newcommand{\OmA}{\Omega_A}
\newcommand{\OmB}{\Omega_B}
\newcommand{\Ah}{\hat{A}}
\newcommand{\omY}{\omega_Y}

\newcommand{\OmY}{\Omega_Y}
\newcommand{\OY}{\shO_Y}

\newcommand{\df}{\mathit{df}}

\newcommand{\newpar}[1]{\subsection{\texorpdfstring{}{}}}
\newcommand{\parref}[1]{\hyperref[#1]{\S\ref*{#1}}}

\newcommand{\OA}{\mathscr{O}_A}
\newcommand{\omX}{\omega_X}
\newcommand{\Pal}{P_{\alpha}}

\begin{document}

\title{Kodaira dimension and zeros of holomorphic one-forms}


\author{Mihnea Popa}
\address{Department of Mathematics, University of Illinois at Chicago,
851 S. Morgan Street, Chicago, IL 60607, USA} 
\email{\tt mpopa@math.uic.edu}

\author{Christian Schnell}
\address{Department of Mathematics, Stony Brook University,
Stony Brook, NY 11794, USA}
\email{\tt cschnell@math.sunysb.edu}

\begin{abstract}
We show that every holomorphic one-form on a smooth complex projective variety of general
type must vanish at some point. The proof uses generic vanishing theory for
Hodge modules on abelian varieties.
\end{abstract}


\maketitle

\section*{Introduction}

\newpar{Two conjectures}

In this article, we use results about Hodge modules on abelian varieties
\cite{PopaSchnell} to prove two conjectures about zeros of holomorphic
one-forms on smooth complex projective varieties. One conjecture was
formulated by Hacon and Kov\'acs \cite{HK} and by Luo and Zhang \cite{LZ}; the
former partly attribute the question to Carrell \cite{Carrell}, and also explain 
why it is natural to consider varieties of general type.

\begin{conjecture}[Hacon-Kov\'acs, Luo-Zhang] \label{conj:main}
If $X$ is a smooth complex projective variety of general type, then the zero locus
\[
	Z(\omega) = \menge{x \in X}{\omega(T_x X) = 0}
\]
of every global holomorphic one-form $\omega \in H^0(X, \OmX^1)$ is
nonempty.
\end{conjecture}

The statement is a tautology in the case of curves; for surfaces, it was proved by
Carrell \cite{Carrell}, and for threefolds by Luo and Zhang \cite{LZ}. It was also
known to be true if the canonical bundle of $X$ is ample \cite{Zhang}, and more
generally when $X$ is minimal \cite{HK}. For varieties that are not necessarily of
general type, Luo and Zhang proposed the following more general version of the
conjecture, in terms of the Kodaira dimension $\kappa(X)$, and proved it in the case
of threefolds \cite{LZ}.

\begin{conjecture}[Luo-Zhang] \label{conj:strong}
Let $X$ be a smooth complex projective variety, and let $W \subseteq H^0(X, \OmX^1)$
be a linear subspace such that $Z(\omega)$ is empty for every nonzero one-form
$\omega \in W$. Then the dimension of $W$ can be at most $\dim X - \kappa(X)$.
\end{conjecture}

\newpar{The main result}

Since every holomorphic one-form on $X$ is the pullback of a holomorphic one-form
from the Albanese variety $\Alb X$, it is natural to consider an arbitrary morphism from
$X$ to an abelian variety, and to ask under what conditions the pullback of a
holomorphic one-form must have a zero at some point of $X$. Our main result
is the following theorem, which implies both of the above conjectures.

\begin{theorem} \label{thm:main}
Let $X$ be a smooth complex projective variety, and let $f \colon X \to A$ be a
morphism to an abelian variety. If $H^0 \bigl( X, \omX^{\tensor d} \tensor \fu L^{-1}
\bigr) \neq 0$ for some integer $d \geq 1$ and some ample line bundle $L$ on $A$,
then $Z(\omega)$ is nonempty for every $\omega$ in the image of $\fu \colon
H^0(A, \Omega_A^1) \to H^0(X, \Omega_X^1)$.
\end{theorem}

The condition in the theorem is equivalent to asking that, up to birational
equivalence, the morphism $f$ should factor through the Iitaka fibration of $(X,
K_X)$.

\newpar{Smooth morphisms to abelian varieties}

One consequence of \conjectureref{conj:strong} is the following.

\begin{corollary} \label{cor:smooth}
If $f \colon X \to A$ is a smooth morphism from a smooth complex projective variety
onto an abelian variety, then $\dim A \leq \dim X - \kappa(X)$.
\end{corollary}

\begin{proof}
Since $f$ is smooth, $Z(\fu \omega)$ is empty for every nonzero $\omega \in
H^1(A, \OmA^1)$. The desired inequality then follows from \conjectureref{conj:strong}
because $\fu$ is injective.
\end{proof}

In particular, there are no nontrivial smooth morphisms from a variety of general
type to an abelian variety; this was proved by Viehweg and Zuo \cite{ViehwegZuo} when
the base is an elliptic curve, and then by Hacon and Kov\'acs \cite{HK}
in general. Together with subadditivity, \corollaryref{cor:smooth} has the following
application.

\begin{corollary}
If $f \colon X \to A$ is a smooth morphism onto an abelian variety, and if all
fibers of $f$ are of general type, then $f$ is birationally isotrivial. 
\end{corollary}

\begin{proof}
We have to show that $X$ becomes birational to a product after a generically finite
base-change; it suffices to prove that $\Var(f) = 0$, in Viehweg's terminology. Since
the fibers of $f$ are of general type, we know from the main result of \cite{Kollar} that
\[
	\kappa(F) + \Var(f) \leq \kappa(X).
\]
On the other hand, $\kappa(X) \leq \dim F$ by \corollaryref{cor:smooth}, and so we
are done.
\end{proof}

Successive versions of this result go back to Migliorini \cite{Migliorini}
(for smooth families of minimal surfaces of general type over an elliptic curve), 
Kov\'acs \cite{Kovacs-smooth} (for smooth families of minimal varieties of general
type over an elliptic curve), and Viehweg-Zuo \cite{ViehwegZuo} (for smooth families
of varieties of general type over an elliptic curve). Over a base of arbitrary
dimension, Kov\'acs \cite{Kovacs-nef} and Zhang \cite{Zhang} have shown that if the
canonical bundles of the fibers is assumed to be ample, the morphism $f$ must
actually be isotrivial.

\newpar{Proof of the conjectures}

Before we get into the details of the proof, we explain how \theoremref{thm:main}
implies \conjectureref{conj:strong}  -- and hence also \conjectureref{conj:main},
which is a special case.

The statement of the conjecture is vacuous when $\kappa(X) = -\infty$, and so we
shall assume from now on that $\kappa(X) \geq 0$. Let $\mu \colon X' \to X$ be a
birational modification of $X$, such that $g \colon X' \to Z$ is a smooth model for
the Iitaka fibration. The general fiber of $g$ is then a smooth projective variety of
dimension $\delta(X) = \dim X - \kappa(X)$ and Kodaira dimension zero, and therefore
maps surjectively to its own Albanese variety \cite{Kawamata}*{Theorem~1}. By a
standard argument \cite{Kawamata}*{Proof of Theorem~13},
the image in $\Alb X$ of every fiber of $g$ must be a translate of a single
abelian variety. Letting $A$ denote the quotient of $\Alb X$ by this abelian variety,
we obtain $\dim A \geq \dim H^0(X, \Omega_X^1) - \delta(X)$, and the following
commutative diagram:
\begin{equation} \label{diag:Iitaka}
\begin{tikzcd}
	X' \dar{g} \rar{\mu} & X \dar{f} \\
	Z \rar & A
\end{tikzcd}
\end{equation}
By construction, $Z$ dominates $f(X)$, and so we conclude from \theoremref{thm:main}
that all holomorphic one-forms in the image of $\fu \colon H^0(A, \Omega_A^1) \to
H^0(X, \Omega_X^1)$ have a nonempty zero locus. Since this subspace has codimension at
most $\delta(X)$ in the space of all holomorphic one-forms, we see that
\conjectureref{conj:strong} must be true.

\section*{Proof of the theorem}

\newpar{From Higgs bundles to Hodge modules}

Let $f \colon X \to A$ be a morphism from a smooth projective
variety to an abelian variety. We define $V = H^0(A, \Omega_A^1)$, and introduce the
set
\[
	Z_f = \menge{(x, \omega) \in X \times V}{(\fu \omega)(T_x X) = 0}.
\]
To prove \theoremref{thm:main}, we have to show that the projection from $Z_f$ to
$V$ is surjective; what we will actually show is that $(f \times \id)(Z_f)
\subseteq A \times V$ maps onto $V$.

\newpar{Summary of our method}

In our proof, we use techniques from generic vanishing theory and from Saito's theory
of Hodge modules. To motivate the use of generic vanishing theory, suppose
for a moment that one could take $d = 1$ in \theoremref{thm:main}; in other words,
suppose that there was a nontrivial morphism
\[
	\fu L \to \omX.
\]
By one of the generic vanishing results due to Green and Lazarsfeld
\cite{GreenLazarsfeld}*{Theorem~3.1}, the existence of a holomorphic one-form $\fu
\omega$ without zeros implies that $H^0(X, \omX \tensor \fu \Pal) = 0$ for general
$\alpha \in \Pic^0(A)$. This leads to the conclusion that
\begin{equation} \label{eq:outcome}
	H^0(A, L \tensor \Pal \tensor \fl \OX) = 0,
\end{equation}
which contradicts the fact that $L$ is ample, because a general translate of an ample
divisor does not contain $f(X)$.

In reality, only some large power of $\omX \tensor \fu L^{-1}$ will have a section --
after an \'etale cover that makes $L$ sufficiently divisible -- and so we are forced
to work on a resolution of singularities of a $d$-fold branched covering of $X$. Now
the additional singularities in the morphism to $A$ prevent the simple argument from
above from going through. To fix this problem, we use an extension of generic
vanishing theory to Hodge modules, developed in \cite{PopaSchnell}. With the
help of this tool, we can still obtain \eqref{eq:outcome} when there is a holomorphic
one-form $\fu \omega$ without zeros.

\newpar{The method of Viehweg and Zuo}
\label{par:VZ}

Our method owes a lot to the paper \cite{ViehwegZuo}, in which Viehweg and Zuo show,
among other things, that there are no smooth morphisms from a complex projective
variety $X$ of general type to an elliptic curve $E$. To orient the reader, we shall
briefly recall a key step in their proof \cite{ViehwegZuo}*{Lemma~3.1}; a similar
technique also appears in the work of Kov\'acs \cites{Kovacs-smooth,Kovacs-log}.
Using a carefully chosen resolution of singularities of a certain branched covering
of $X$, Viehweg and Zuo produce a logarithmic Higgs bundle $\bigoplus
\mathcal{E}^{p,q}$ on the elliptic curve $E$, whose Higgs field
\[
	\theta^{p,q} \colon \mathcal{E}^{p,q} 
		\to \mathcal{E}^{p-1,q+1} \tensor \Omega_E^1(\log S')
\]
has logarithmic poles along a divisor $S'$; by construction, $S'$ contains the set of
points $S$ where the morphism from $X$ to $E$ is singular. They also construct a
Higgs subbundle $\bigoplus \mathcal{F}^{p,q}$ with two properties: the restriction
of the Higgs field satisfies
\[
	\theta^{p,q} \bigl( \mathcal{F}^{p,q} \bigr) 
		\subseteq \mathcal{F}^{p-1,q+1} \tensor \Omega_E^1(\log S),
\]
and $\mathcal{F}^{r,0}$ is an ample line bundle (where $r = \dim X - 1$). Since any
subbundle of $\ker \theta^{p,q}$ has degree $\leq 0$, it follows that $S$ cannot be
empty: otherwise, the image of $\mathcal{F}^{r,0}$ under some iterate of the Higgs
field would be an ample subbundle of $\ker \theta^{p,q}$.

\newpar{From Higgs bundles to Hodge modules}

The associated graded of a Hodge module may be thought of as a generalization of
a Higgs bundle. Recall that a Hodge module on a
smooth complex algebraic variety $X$ is a very special kind of filtered $\Dmod$-module
$(\Mmod, F)$; in particular, $\Mmod$ is a regular holonomic left $\Dmod_X$-module,
and $F_{\bullet} \Mmod$ is a good increasing filtration by $\OX$-coherent subsheaves.
The associated graded 
\[
	\gr_{\bullet}^F \! \Mmod = \bigoplus_{k \in \ZZ} F_k \Mmod \big\slash F_{k-1} \Mmod
\]
is coherent over the symmetric algebra $\Sym \shT_X$. It therefore determines a
coherent sheaf $\gr^F \! \Mmod$ on the cotangent bundle $T^{\ast} X$, whose support
is the so-called characteristic variety $\Ch(\Mmod)$ of the $\Dmod$-module. A longer
summary of Saito's theory can be found in \cite{PopaSchnell}*{Sections~2.1--2},
and of course in the introduction to \cite{Saito}.

\newpar{Another reason for Hodge modules}

For later use, we point out a connection between \theoremref{thm:main}
and Hodge modules, having to do with the properties of the set $(f \times
\id)(Z_f)$.  
The structure sheaf $\OX$ is naturally a left $\Dmod_X$-module; the direct image
functor for $\Dmod$-modules takes it to a complex $\fp \OX$ of regular holonomic
$\Dmod_A$-modules. According to Kashiwara's estimate for the behavior of the
characteristic variety,
\begin{equation} \label{eq:Kashiwara}
	\Ch \bigl( \fp \OX \bigr) \subseteq
		(f \times \id) \bigl( \df^{-1}(0) \bigr) = (f \times \id)(Z_f),
\end{equation}
where the notation is as in the following diagram:
\begin{equation} \label{eq:Japanese}
\begin{tikzcd}
X \times V \dar{f \times \id} \rar{\df} & T^{\ast} X \\
A \times V
\end{tikzcd}
\end{equation}
One consequence of Saito's theory is that $\OX$, equipped with the obvious filtration
($\gr_0^F \OX = \OX$), is actually a Hodge module.
Because $f$ is projective, $\fp(\OX, F)$ (with the induced filtration) is then a
complex of Hodge modules on $A$, and Saito's Decomposition Theorem
\cite{Saito}*{Th\'eor\`eme~5.3.1} gives us a non-canonical splitting
\begin{equation} \label{eq:decomposition}
	\fp(\OX, F) \simeq \bigoplus_{i \in \ZZ} \shH^i \fp(\OX, F) \decal{-i}
\end{equation}
in the derived category of filtered $\Dmod$-modules. This means that the set $(f \times
\id)(Z_f)$ contains the characteristic varieties of the Hodge modules
$\shH^i \fp(\OX, F)$.

\newpar{The idea behind our proof} 
\label{par:idea}

The above considerations suggest a way to generalize the construction in
\cite{ViehwegZuo} to the setting of \theoremref{thm:main}. Namely,
suppose that we manage to find a Hodge module $(\Mmod, F)$ on the abelian
variety $A$, and a graded $\Sym \shT_A$-submodule 
\begin{equation} \label{eq:objects}
	\shF_{\bullet} \subseteq \gr_{\bullet}^F \! \Mmod.
\end{equation}
Denote by $\shF$ and $\gr^F \! \Mmod$ the associated coherent sheaves on $T^{\ast} A
= A \times V$, and suppose that the following three conditions are satisfied:
\begin{enumerate}[label=(\arabic{*}), ref=(\arabic{*})]
\item \label{en:condition-1}
There is a morphism $h \colon Y \to A$ from a smooth projective variety, such that
$(\Mmod, F)$ is a direct summand of some $\shH^i \hp(\OY, F)$.
\item \label{en:condition-2}
The support of $\shF$ is contained in the set $(f \times \id)(Z_f) \subseteq A
\times V$.
\item \label{en:condition-3}
For some $k \in \ZZ$, the sheaf $\shF_k$ is isomorphic to $L \tensor \fl \OX$, where
$L$ is an ample line bundle on $A$.
\end{enumerate}
If that is the case, we can use the generic vanishing theory for Hodge
modules that we developed in \cite{PopaSchnell} to show that $Z_f$ projects
onto $V$.

\begin{proposition} \label{prop:surjective}
If a Hodge module $(\Mmod, F)$ and a graded $\Sym \shT_A$-module
$\shF_{\bullet}$ with the above properties exist, then the projection from $Z_f$ to $V$
must be surjective.
\end{proposition}

\begin{proof}
Let $P$ be the normalized Poincar\'e bundle on $A \times \Ah$. Using its pull-back to
$A \times \Ah \times V$ as a kernel, we define the Fourier-Mukai transform of $\gr^F
\! \Mmod$ to be the complex of coherent sheaves
\[
	E = \derR \Phi_P(\gr^F \! \Mmod) 
		= \derR (p_{23})_{\ast} \Bigl( p_{13}^{\ast} \bigl( \gr^F \! \Mmod \bigr) 
			\tensor p_{12}^{\ast} P \Bigr)
\] 
on $\Ah \times V$. When the Hodge module $(\Mmod, F)$ comes from a morphism to an
abelian variety as in \ref{en:condition-1}, the complex $E$ has special properties
\cite{PopaSchnell}*{Section~4.5}:
\begin{aenumerate}
\item \label{en:GVT-a}
$E$ is a \define{perverse coherent sheaf}, meaning that its cohomology sheaves
$\shH^{\ell} E$ are zero for $\ell < 0$, and are supported in codimension at least
$2\ell$ otherwise.
\item \label{en:GVT-b}
The union of the supports of all the higher cohomology sheaves of $E$ is a finite
union of translates of triple tori in $\Ah \times V$.
\item \label{en:GVT-c}
The dual complex $\derR \shHom(E, \shO)$ has the same properties.
\end{aenumerate}
Here a triple torus, in Simpson's terminology, means a subset of the form 
\[
	\im \Bigl( \varphi^{\ast} \colon \hat{B} \times H^0(B, \OmB^1)
		\to \Ah \times H^0(A, \OmA^1) \Bigr),
\]
where $\varphi \colon A \to B$ is a morphism to another abelian variety. Roughly
speaking, \ref{en:GVT-a} is derived using a vanishing theorem by Saito
\cite{PopaSchnell}*{Lemma~2.5}, whereas \ref{en:GVT-b} is derived using a result by Arapura
and Simpson about cohomology support loci of Higgs bundles \cite{Arapura}, with the
help of the decomposition in \eqref{eq:decomposition}.

It follows that the $0$-th cohomology sheaf of the complex $E$ is locally free
outside a finite union of translates of triple tori of codimension at least two.
Indeed, the locus where $\shH^0 E$ is not locally free is
\[
	\bigcup_{\ell \geq 1} \Supp R^{\ell} \shHom \bigl( \shH^0 E, \shO \bigr) 
		\subseteq \bigcup_{\ell \geq 1} \Supp R^{\ell} \shHom(E, \shO) \cup 
		\bigcup_{\ell \geq 1} \Supp \shH^{\ell} E,
\]
which is of the asserted kind because $\derR \shHom(E, \shO)$ also satisfies
\ref{en:GVT-a} and \ref{en:GVT-b}. Since the higher cohomology
sheaves of $E$ are supported in a finite union of translates of triple tori of
codimension at least two, the restriction of the complex $E$ to the subspace
$\{\alpha\} \times V$ is a single locally free sheaf for general $\alpha \in \Ah$. 

Translating this back into a statement on $A \times V$ involves a few simple 
manipulations with the formula for the Fourier-Mukai transform; the result is that
\[
	p_{2 \ast} \bigl( \pu_1 \Pal \tensor \gr^F \! \Mmod \bigr)
\]
is a locally free sheaf on $V$. Here $p_1 \colon A \times V \to A$ and $p_2 \colon A
\times V \to V$ are the projections to the two factors, and $\Pal$ is the line bundle
corresponding to $\alpha \in \Ah$.

Now suppose that the projection from $Z_f$ to $V$ was \emph{not} surjective. Because
the support of $\shF$ is contained in $(f \times \id)(Z_f)$ by \ref{en:condition-2},
the subsheaf
\[
	p_{2\ast} \bigl( \pu_1 \Pal \tensor \shF \bigr)
		\subseteq p_{2\ast} \bigl( \pu_1 \Pal \tensor \gr^F \! \Mmod \bigr)
\]
would then be torsion, and therefore zero. As $V$ is a vector space, we are
forced to the conclusion that the corresponding graded $\Sym V^{\ast}$-module
\[
	\bigoplus_{k \in \ZZ} H^0 \bigl( A, \Pal \tensor \shF_k \bigr),
\]
is also zero. But this contradicts our assumption in \ref{en:condition-3} that $\shF_k
\simeq L \tensor \fl \OX$ for some $k \in \ZZ$. In fact, if $\Pal \tensor L \tensor
\fl \OX$ has no global sections, then every global section of $\Pal \tensor L$ has to
vanish set-theoretically along $f(X)$; but then $f(X)$ is contained in a general
translate of an ample divisor, which is absurd.
\end{proof}

\newpar{Producing sections}

The remainder of the paper is devoted to constructing the two objects in
\eqref{eq:objects} under the assumptions of \theoremref{thm:main}. From now on, we
let $X$ be a smooth complex projective variety of dimension $n$ and Kodaira dimension
$\kappa(X) \geq 0$. We also assume that we have a morphism $f \colon X \to A$ to an
abelian variety, such that $\omX^{\tensor d} \tensor \fu L^{-1}$ has a section for
some $d \geq 1$ and some ample line bundle $L$ on $A$. Using the geometry of abelian
varieties, this can be improved as follows.

\begin{lemma}
After a finite \'etale base change on $A$, we can find an ample line bundle $L$
such that the $d$-th power of $\omega_X \tensor \fu L^{-1}$ has a section for some $d
\geq 1$.
\end{lemma}

\begin{proof}
Fix an ample line bundle $L_1$ on $A$, and choose $d \geq 1$ such that
$\omX^{\tensor d} \tensor \fu L_1^{-1}$ has a section. Let $\decal{2d} \colon A \to A$
denote multiplication by $2d$; then $\decal{2d}^{\ast} L_1 \simeq L^{\tensor d}$ for some
ample line bundle $L$ on $A$, which clearly does the job.
\end{proof}

Since the conclusion of \theoremref{thm:main} is unaffected by finite \'etale
morphisms, we may assume for the remainder of the argument that the $d$-th power
of the line bundle $B = \omega_X \tensor \fu L^{-1}$ has a nontrivial section; for
the sake of convenience, we shall take $d \geq 1$ to be the smallest integer with this
property.

\newpar{Constructing a branched cover}

A nontrivial section $s \in H^0 \bigl( X, B^{\tensor d} \bigr)$ defines a branched
covering $\pi \colon X_d \to X$ of degree $d$, unramified outside the divisor $Z(s)$;
see \cite{EV}*{\S3} for details. Since $d$ is minimal, $X_d$ is irreducible; let $\mu
\colon Y \to X_d$ be a resolution of singularities that is an isomorphism over the
complement of $Z(s)$, and define $\varphi = \pi \circ \mu$ and $h = f \circ \varphi$.
The following commutative diagram shows all the relevant morphisms:
\begin{equation} 
\begin{tikzcd}
Y \rar{\mu} \arrow[bend right=20]{drr}{h} \arrow[bend left=40]{rr}{\varphi} 
		& X_d \rar{\pi} & X \dar{f} \\
 & & A
\end{tikzcd}
\end{equation}
By construction, $X_d$ is embedded in the total space of the line bundle $B$, and
so the pullback $\piu B$ has a tautological section; the induced morphism $\varphiu
B^{-1} \to \OY$ is an isomorphism over the complement of $Z(s)$. After composing it
with $\varphiu \OmX^k \to \OmY^k$, we obtain for every $k = 0, 1, \dotsc, n$ an
injective morphism
\[
	\varphiu \bigl( B^{-1} \tensor \OmX^k \bigr) \to \OmY^k,
\]
which is actually an isomorphism over the complement of $Z(s)$. Pushing forward to
$X$, and using the fact that $\OX \to \varphil \OY$ is injective, we find that the
morphisms
\begin{equation} \label{eq:morphism}
	B^{-1} \tensor \OmX^k \to \varphil \OmY^k
\end{equation}
are also injective.

\newpar{Building a morphism of complexes}

Let $S = \Sym V^{\ast}$ be the symmetric algebra on the dual of $V = H^0(A,
\Omega_A^1)$, and consider the complex of graded $\OX \tensor S$-modules
\[
	C_{X, \bullet} = \Bigl\lbrack 
		\OX \tensor S_{\bullet-g} \to \OmX^1 \tensor S_{\bullet-g+1} \to \dotsb
			\to \OmX^n \tensor S_{\bullet-g+n} \Bigr\rbrack,
\]
placed in cohomological degrees $-g, \dotsc, 0$, where $g = \dim A$ and $n = \dim X$.
The differential in the complex is induced by the evaluation morphism $V \tensor \OX
\to \OmX^1$. Concretely, let $\omega_1, \dotsc, \omega_g \in V$ be a basis, and
denote by $s_1, \dotsc, s_g \in S_1$ the dual basis; then the formula for the
differential is
\[
	\OmX^p \tensor S_{\bullet-g+p} \to \OmX^{p+1} \tensor S_{\bullet-g+p+1}, \quad 
		\theta \tensor s \mapsto 
			\sum_{i=1}^g \bigl( \theta \wedge \fu \omega_i \bigr) \tensor s_i s.
\]
We use similar notation on $Y$ as well.

\begin{lemma} \label{lem:morphism}
There is a morphism of complexes of graded $\OA \tensor S$-modules
\[
	\derR \fl \bigl( B^{-1} \tensor C_{X, \bullet} \bigr) \to \derR \hl C_{Y, \bullet},
\]
induced by the individual morphisms in \eqref{eq:morphism}.
\end{lemma}

\begin{proof}
The morphisms in \eqref{eq:morphism} commute with the differentials in the two
complexes because $\varphiu \bigl( \fu \omega \bigr) = \hu \omega$ for every $\omega
\in V$.
\end{proof}

\newpar{Controlling the support}

We denote by $C_X$ the complex of coherent sheaves on $X \times V$ associated with
the complex of graded $\OX \tensor S$-modules $C_{X, \bullet}$. 

\begin{lemma} \label{lem:support}
The support of $C_X$ is equal to $Z_f \subseteq X \times V$.
\end{lemma}

\begin{proof}
Let $p_1 \colon X \times V \to X$ denote the first projection; then
\[
	C_X = \Bigl\lbrack 
		\pu_1 \OX \to \pu_1 \OmX^1 \to \dotsb \to \pu_1 \OmX^n \Bigr\rbrack,
\]
with differential induced by the tautological section of $\pu_1 \OmX^1$. This shows
that $C_X$ is equal to the pullback of the Koszul resolution for the structure sheaf of
the zero section in $T^{\ast} X$ via the morphism $\df \colon X \times V \to
T^{\ast} X$. In particular, $\Supp C_X$ is equal to $\df^{-1}(0) = Z_f$, in the
notation of \eqref{eq:Japanese}.
\end{proof}

\newpar{Relationship with Hodge modules}

The reason for introducing $C_{X, \bullet}$ is that it is closely related to the
direct image $\fp(\OX, F)$. In fact, we have the following refinement of Kashiwara's
estimate \eqref{eq:Kashiwara}.

\begin{lemma} \label{lem:Laumon}
The associated graded of the Hodge module $\shH^i \fp(\OX, F)$ is
\[
	\gr_{\bullet}^F \Bigl( \shH^i \fp(\OX, F) \Bigr) \simeq R^i \fl C_{X, \bullet}.
\]
\end{lemma}

\begin{proof}
This is proved in \cite{PopaSchnell}*{Proposition~2.11}. In a nutshell, a result
by Laumon gives the isomorphism $\gr_{\bullet}^F \fp(\OX, F) \simeq \derR \fl C_{X,
\bullet}$ in the derived category of graded $\OA \tensor S$-modules; to deduce the
assertion about individual cohomology sheaves, one has to use the fact that 
the complex $\fp(\OX, F)$ is strict \cite{Saito}*{Th\'eor\`eme~5.3.1}. This ensures
that every $\shH^i \fp(\OX, F)$ is again a filtered $\Dmod$-module; a priori, it is
only an object of a larger abelian category, due to the fact that the derived category
of filtered $\Dmod$-modules is the derived category of an exact category. Loosely
speaking, this means that, for Hodge modules, taking the associated graded commutes
with direct images; arbitrary filtered $\Dmod$-modules do not have this property.
\end{proof}

\newpar{Constructing a suitable Hodge module}

We are now in a position to carry out the construction suggested in \parref{par:idea}.
For the Hodge module $(\Mmod, F)$, we take $\shH^0 \hp(\OY, F)$; it is really a Hodge
module because the morphism $h$ is projective \cite{Saito}*{Th\'eor\`eme~5.3.1}. As
explained in \lemmaref{lem:Laumon}, one has an isomorphism
\[
	\gr_{\bullet}^F \! \Mmod \simeq R^0 \hl C_{Y, \bullet}
\]
as graded modules over $\Sym \shT_A = \OA \tensor S$. Now define
$\shF_{\bullet}$ to be the image of $R^0 \fl \bigl( B^{-1} \tensor C_{X, \bullet}
\bigr)$ in $R^0 \hl C_{Y, \bullet}$, using the the morphism provided by
\lemmaref{lem:morphism}. In this way, we get a graded $\Sym \shT_A$-submodule of
$\gr_{\bullet}^F \! \Mmod$ as in \eqref{eq:objects}.

\newpar{Concluding the proof}

It remains to check that $(\Mmod, F)$ and $\shF_{\bullet}$ have all the required properties.

\begin{proposition} \label{prop:construction}
$(\Mmod, F)$ and $\shF_{\bullet}$ satisfy the conditions in
\ref{en:condition-1}--\ref{en:condition-3}.
\end{proposition}

\begin{proof}
It is obvious from the construction that \ref{en:condition-1} holds. To prove
\ref{en:condition-2}, note that $\shF$ is by definition a quotient of the coherent sheaf
\[
	R^0 (f \times \id)_{\ast} \bigl( p_1^{\ast} B^{-1} \tensor C_X \bigr).
\]
The complex in parentheses is supported in the set $Z_f$ by \lemmaref{lem:support};
consequently, the support of $\shF$ is contained in $(f \times \id)(Z_f)$, as
required. To finish up, we shall argue that \ref{en:condition-3} is true when $k =
g-n$. We clearly have $C_{X, k} = \omX$ and $C_{Y, k} = \omY$, and the morphism $\fu
L = B^{-1} \tensor \omX \to \varphil \omY$ in \eqref{eq:morphism} is injective. After
pushing forward to the abelian variety $A$, we find that the resulting morphism
\[
	L \tensor \fl \OX \simeq \fl \fu L \to \hl \omY
\]
is still injective. But $\gr_k^F \! \Mmod \simeq \hl \omY$, and so $\shF_k$ is 
isomorphic to $L \tensor \fl \OX$. 
\end{proof}

Now \propositionref{prop:surjective} shows that $Z_f$ projects onto $V$; this means
that every holomorphic one-form in the image of $\fu \colon H^0(A, \Omega_A^1) \to
H^0(X, \OmX^1)$ has a nonempty zero locus.  We have proved
\theoremref{thm:main}, \conjectureref{conj:strong}, and \conjectureref{conj:main}.

\newpar{Comparison with Viehweg-Zuo}

To illustrate how our proof works, we shall briefly compare it with the argument by
Viehweg and Zuo (sketched in \parref{par:VZ}); either method shows that there are no
smooth morphisms $f \colon X \to E$ from a variety of general type to an elliptic
curve. It is easy to see that $(f \times \id)(Z_f)$ dominates $V = H^0(E,
\Omega_E^1)$ if and only if $f$ has at least one singular fiber, so let us quickly
see how either method leads to a contradiction. 

In the case of Viehweg and Zuo, $f$ smooth implies that $S = \emptyset$, and hence
that the iterates under the Higgs field
\[
	\bigl( \theta^{r,0} \circ \theta^{r-1,1} \circ \dotsb \circ \theta^{p,q} \bigr)
		(\shF^{r,0}) \subseteq \shE^{p-1,q+1}
\]
of the ample line bundle $\shF^{r,0}$ are all nonzero (because subbundles of
$\ker \theta^{p,q}$ have degree $\leq 0$); but this is not possible because
$\theta^{0,r} = 0$. In other words, smoothness of $f$ and positivity of $\shF^{r,0}$
would prevent the Higgs field from being nilpotent, contradicting the fact that it
is always nilpotent.

In our case, let $\partial \in H^0(E, \shT_E)$ be a nonzero global vector
field. The method used in \propositionref{prop:surjective} produces a nonzero
graded $\CC \lbrack \partial \rbrack$-submodule
\[
	\bigoplus_{k \geq g-n} H^0(X, \Pal \tensor \shF_k)	
		\subseteq \bigoplus_{k \geq g-n} H^0 \bigl( X, \Pal \tensor \gr_k^F \! \Mmod \bigr),
\]
which is $\partial$-torsion if $(f \times \id)(Z_f) = E \times \{0\}$; but this is not
possible because the ambient module is free over $\CC \lbrack \partial \rbrack$. In
other words, smoothness of $f$ and positivity of $\shF_{g-n}$ would create a
nontrivial torsion submodule, contradicting the fact that the larger module is always
free.

More generally, our branched covering construction is nearly identical to the one in
\cite{ViehwegZuo}, except that we do not need to be as careful in choosing the initial
section of $B^{\tensor d}$ or the resolution of singularities; this is due to the
power of Saito's theory. The idea of having two objects -- one with good properties
coming from Hodge theory, the other encoding the singularities of the morphism $f$ --
also comes from Viehweg and Zuo. But the actual mechanism behind the
proof of \propositionref{prop:surjective} based on generic vanishing theory is
different, because Higgs bundles and Hodge modules have different properties.

\section*{Acknowledgements}

\newpar{Personal}

We thank Lawrence Ein, S\'andor Kov\'acs, and Mircea Musta\c{t}\u{a} for several
useful discussions, and especially Ein for discovering a serious mistake in an earlier
attempt to prove the conjectures. We also thank an anonymous referee for many
comments that have greatly improved  the presentation.

\newpar{Funding}

During the preparation of this paper, Mihnea Popa was partially supported by 
grant DMS-1101323 from the National Science Foundation, and Christian Schnell by
grant DMS-1331641.

\end{document}